\documentclass[12pt,a4paper]{article}

\usepackage{verbatim}

\usepackage{amsmath,amsthm}
\usepackage{epsfig}
\usepackage{epstopdf}
\usepackage{amssymb}
\usepackage{epic}
\usepackage{eepic}

\def\ZZ{{\mathbb Z}}

\def\oo{\infty}

\def\g{\gamma}
\def\G{\Gamma}
\def\a{\alpha}
\def\b{\beta}
\def\d{\delta}
\def\l{\lambda}

\newcommand{\eps}{\varepsilon}
\newcommand{\seq}{\subseteq}
\renewcommand{\k}{\kappa}
\newcommand{\s}{\sigma}
\newcommand{\ru}{\mathrm{u}}
\newcommand{\var}{\mathrm{Var}}
\newcommand{\one}{\hbox{\rm 1\kern-.27em I}}

\theoremstyle{definition}
\newtheorem{theorem}{Theorem}[section]
\newtheorem{lemma}[theorem]{Lemma}

\newtheorem{definition}[theorem]{Definition}

\newtheorem{remark}[theorem]{Remark}

\author{Jakob E. Bj\"ornberg, Erik I. Broman}

\title{Coexistence and non-coexistence\\ of Markovian 
viruses and their hosts}

\begin{document}
\maketitle
\begin{abstract}
The possibility
of coexistence of two competing populations is a classical question 
which dates back to the earliest `predator-prey' models. In this paper
we study this question
in the context of a model for the spread of a virus infection
in a population of healthy cells, introduced in~\cite{BBBN}.
The infected cells may be seen as a population
of `predators' and the healthy cells as a population
of `prey'.  We show that, depending on the parameters
defining the model, there may or may not be coexistence of
the two populations, and we give precise criteria for this.
\end{abstract}

\section{Introduction} \label{intro_sec}
We start by giving an informal description of the model studied in this paper.
It is a two-dimensional Markov process $(X(t),Y(t))_{t\geq 0}$, where
$X(t)$ is the number of `healthy cells' at time $t$, and
$Y(t)$ is the number of `infected cells' (i.e.\ cells
having virus in them). Both components 
$(X(t))_{t\geq0}$  and $(Y(t))_{t\geq0}$ behave in many ways like
branching processes, although there are dependencies between them.
A healthy cell is replaced by a random number of new healthy
cells at rate 1.  This random number is independent of
other events and drawn from a distribution
$(p_k)_{k\geq0}$;  thus the rate at which 
a healthy cell is replaced by $k$ healthy cells is
$p_k$.  Infected cells are also replaced by $k$ new (infected) cells at 
rate $p_k$ if $k\geq 1$ while they are replaced by 0
new cells (die) at the higher rate $p_0+\lambda$. 
Here $\lambda\geq 0$ is a parameter
that reflects the negative impact of the virus on the host's
lifelength.  When an infected cell dies (i.e.\! is replaced
by 0 new cells), it converts a random number of healthy cells into
infected cells.  The biological motivation is that when infected cells
die they burst (lyse) and release `free virions'
which enter a random number of 
healthy cells, thus infecting them. 
The number of conversions is independent of all other
events, and is drawn from a distribution $(\g_k)_{k\geq0}$.
Hence, the processes $(X(t))_{t\geq0}$ and $(Y(t))_{t\geq 0}$ interact
in that $(Y(t))_{t\geq0}$ `feeds' upon $(X(t))_{t\geq0}.$
The model is defined in detail in
Section~\ref{def_sec}. Also, we refer the interested reader to \cite{BBBN}
for a biological motivation of the model. We will sometimes simply write 
$X$ or $Y$ as a shorthand for $(X(t))_{t\geq0}$ and $(Y(t))_{t\geq0}$ respectively.

As described, the model is in essence a pair of interacting branching processes. 
Markov branching processes with interaction have been much studied, see for 
instance~\cite{kalinkin02} and the references within. 
The main purpose of this paper is the study of coexistence of the two, competing, populations 
$X$ and $Y$. 
Similar types of questions have been studied in many contexts. 
One recent example is the so-called 
two-type-Richardson model. This can be informally described as follows. Consider the graph ${\mathbb Z}^d,$ and 
let the two infections (red and blue) start with only one individual each. A site is infected
by the red (blue) process at a rate which equals the infection parameter 
$\lambda_r$ ($\lambda_b$) times the number 
of neighbours infected by the red (blue) process. 
Further, if a site gets infected by the red infection it stays red 
forever and similarly if it is infected by the blue infection. The main question is if they can 
coexist, i.e. if there will be two unbounded components of red and blue sites,
see for instance ~\cite{hp98,garet05,hoffman05}.

In \cite{BBBN}, much focus was on the study of the 
\emph{extinction probability} $\eta$ of the
infected process $(Y(t))_{t\geq 0}$.
There, $\eta$ was taken as an indicator
of the `evolutionary fitness' of the virus. The main result was that for fixed
$(p_k)_{k\geq0}$ and $(\g_k)_{k\geq0}$ satisfying $\g_0=0$
the extinction probability
$\eta$ is maximized when $\l=0$.
In fact, it was shown that $\eta$ is increasing in $\l$.
The main result of this paper concerns the 
\emph{coexistence probability} $\zeta.$
\begin{definition} \label{defcoex}
We call
\[
\zeta=P(X(t)Y(t)\geq1\mbox{ for all }t\geq 0).
\]
the coexistence probability of $X(t),Y(t).$ 
\end{definition}
Of course, $\zeta$ depends on the parameters used to define the
process, but we suppress this dependence in the notation.  Introducing the stopping
time $T_\ru=\inf\{t\geq0:X(t)Y(t)=0\}$, we have that
$\zeta=P(T_\ru=\infty)$. The relevance of coexistence in the study of
$\eta$ will be discussed in Section \ref{sec_app}.

The proof of our main result uses two auxiliary branching
processes $\hat X (t)$ and $\hat Y (t)$
defined and discussed in detail in Section \ref{auc_sec}.
Informally, $\hat X (t)$ is a process distributed 
as $X(t)$ without the influence of $Y(t),$ i.e.\!
$Y(0)=0.$ Furthermore,  $\hat Y (t)$ is a process distributed 
as $Y(t)$ with an infinite supply of
healthy cells, i.e. $X(0)=\infty.$
Our main result is formulated in terms of the so-called 
malthusian parameters for these processes,
denoted by $\a$ and $\b$ for $\hat X(t)$ 
and $\hat Y(t)$ respectively (see Section \ref{auc_sec}).
It turns out that 
$\a=\sum_{k=0}^\infty kp_k-1$
and $\b=\a +p_0\sum_{k=0}^\infty k\g_k+\l(\sum_{k=0}^\infty k\g_k-1).$

\begin{theorem}\label{coexist_thm}
For arbitrariy initial conditions $X(0),Y(0)\geq 1$ and
offspring distributions $(p_k)_{k\geq0}$ with finite second moment, 
the coexistence probability is positive if and only if $\a>\b>0$.
\end{theorem}
\begin{remark}
Note that coexistence is only possible 
if $\g_0>0$, because otherwise $\b>\a$. This follows from the 
expressions for $\a$ and $\b$ and Theorem \ref{coexist_thm}: 
if $\g_0=0$ then $Y$ cannot die out as long as $X$ survives.
Hence, 
there is then almost surely a time $t>0$
such that $X(t)=0$ and $Y(t)\neq0$.
(This result was announced in~\cite{BBBN} as part of
Proposition~3.2.)
\end{remark}
Theorem \ref{coexist_thm} establishes, under a second moment 
condition, for which values of $\alpha$ and $\beta$ we can have coexistence.
Our next result strengthens the second part of 
Theorem \ref{coexist_thm}. Recall that $T_\ru=\inf\{t>0:X(t)Y(t)=0\}.$
\begin{theorem} \label{thmfinexp}
For offspring distributions $(p_k)_{k\geq0}$ with finite second 
moment, and for any choice of $\a < \b$ we have that
$E[T_\ru]<\oo$.
\end{theorem}
We have not been able to establish 
in general if $T_\ru$ has finite 
or infinite expectation when $\a=\b$.  (But see
Remark~\ref{inf_rk} for a special case.)

On the way to proving that coexistence is indeed possible
(when $\a>\b$) we use general facts about
order statistics and trimmed sums, see Lemma~\ref{trim_lem}.
The second part of
that lemma is an interesting application of 
Harris' inequality~\cite{harris60} 
to bound the variance of
a trimmed sum, which we have not found in the
literature.

In the two-type Richardson model mentioned above, coexistence is conjectured
to hold if and only if $\lambda_r=\lambda_b.$ The `if' condition has been
established, see~\cite{hp98, garet05,hoffman05}, while \cite{hp00} makes progress
on the `only if' condition. In fact, the model studied here is closer to the following 
variant of the two-type  Richardson model. 
If a site is infected by the blue process, it changes color if the red process attempts to infect
it, while if a site is infected by the red process it stays so forever. That is, a red site is 
immune to the blue process while a blue site is not immune to the red infection.
Analogy with the model studied in this paper suggests that there can
then be coexistence if $\l_b>\l_r$, but not if $\l_b<\l_r$.

We end this section with an outline of the rest of the paper. In Section~\ref{def_sec}
we give a precise definition of the model. In Section~\ref{sec_prel} we state and prove 
preliminary results needed in the proofs of our main results. In Sections~\ref{sec_thm1}
and~\ref{sec_thm2} we prove Theorems~\ref{coexist_thm} and~\ref{thmfinexp} respectively.
Finally we discuss some applications of these results in Section~\ref{sec_app}.

\section{Definition}\label{def_sec}
Let $(p_k)_{k\geq0}$ and $(\g_k)_{k\geq0}$ be probability
distributions on the nonnegative integers, and let $\l\geq0$.
We exclude the (degenerate) case when $p_1=1$; 
in fact the reader may for convenience
assume that $p_1 = 0,$ since this only amounts to a time-change.

The continuous--time Markov chain $(X(t),Y(t))_{t\geq 0}$, 
taking values in $\ZZ^2_+$, was  informally described
in Section~\ref{intro_sec}.  To recapitulate the main points,
each healthy cell is replaced by $k\geq0$ new healthy cells
at rate $p_k$.  Being replaced by $k=0$ new cells corresponds
to dying.  Each infected cell is replaced  by $k\geq1$
new infected cells at rate $p_k$.  When an infected cell
dies, which occurs at rate $p_0+\l$, a random number
of healthy cells are converted into infected cells.
If $t$ is the time of such an event,
we draw a random variable $\Gamma_t$ from the distribution $(\g_k)_{k\geq0}$
independently of other events. If $\Gamma_t\leq X(t)$ we simply declare 
$\Gamma_t$ of previously healthy cells to be infected, while if
$\Gamma_t > X(t)$ we declare all previously healthy cells to be infected.
To define this process formally, we list the 
different possible jumps in Table~\ref{table1}, where we use the
notation $x\wedge k=\min\{x,k\}$.

\begin{table}[hbt]
\centering
\begin{tabular}{l|l|l|l} 
& Transition from $(x,y)$ to  & Rate & Valid for\\ \hline
(i) & $(x-1+k,y)$ & $x p_k$ & $k\geq 0$\\
(ii)  & $(x,y-1+k)$ & $y p_k$ & $k\geq 1$\\
(iii) & $(x-(x\wedge k),y-1+(x\wedge k))$ &
$y(p_0+\l)\g_k$ & $k\geq0$
\end{tabular}
\caption{Transition rates for the process $(X(t),Y(t))_{t\geq 0}$.
Rates are given for transitions from a state $(x,y)$ and are valid for all $x,y\geq 0$.}\label{table1} 
\end{table}
Note that there may be several transitions in Table~\ref{table1}
 leading to the same state.  In such cases the
correct interpretation is to add the corresponding rates.
An example of this is the transition
$(0,y)\rightarrow (0,y-1)$, which occurs at rate $y(p_0+\l),$
which is the sum over all $k\geq 0$ in (iii).
To avoid trivial cases, we assume
throughout that $X(0),Y(0)\geq 1.$  Biologically it might
be most relevant to consider the case when $p_k=0$ for
$k\geq 3$, but all our results are valid in greater generality,
so we make no such restriction.

We now state some immediate properties of the model.  If it were the
case that $Y(t)=0$, then healthy cells would evolve as a Markov
branching process, with intensity $1$ and offspring distribution
$(p_k)_{k\geq0}$. 
Similarly, if $X(t)=0$ for some $t,$ then $(Y(t+s))_{s\geq 0}$ would
behave like a Markov branching process with the higher intensity
$(1+\lambda)$ and an offspring distribution 
derived from $(p_k)_{k\geq 0}$ by placing more mass
on $k=0$. 
When both $X(t),Y(t)>0,$ as transition rate (iii) tells us,
healthy cells may turn into infected cells. This scenario hence
`helps' the process $(Y(t))_{t\geq 0}$ and 
`hurts' the process $(X(t))_{t\geq 0}$.

\section{Preliminary results} \label{sec_prel}

In this section we establish several lemmas which 
will be used in the proofs of
Theorems~\ref{coexist_thm} and~\ref{thmfinexp}.  
Although their motivation
may not be obvious on a first reading, we find it convenient
to collect all such preliminary results here so as not to
interrupt the flow of the main proofs later.

A note on notation:  we will sometimes write a sum of the form
$\sum_{k=1}^ax_k$ where $a$ is non-integer.  The correct
interpretation is that the sum goes to the integer part 
$\lfloor a\rfloor$ but we prefer to omit the 
$\lfloor \cdot\rfloor$ to keep the notation more readable.
A similar comment applies also in other places throughout the paper.

\subsection{Auxiliary random variables}
\label{auc_sec}

It will at several points be useful to compare $X$ and $Y$ to two
`larger' processes $\hat X$ and $\hat Y$.  Here $\hat X$ may be
thought of as the healthy process in the absence of infection,
and $\hat Y$ as the infected process in an infinite `sea' of healthy
cells.

To be precise, we let $\hat X$ and $\hat Y$ be two branching
processes with lifelength intensities $1$ and $1+\l$, and offspring
distributions $(p_k)_{k\geq 0}$ and $(q_k)_{k\geq0}$, 
respectively, where $(q_k)_{k\geq0}$ is given by
\begin{equation}\label{q_eq}
q_0=\frac{\g_0(p_0+\l)}{1+\l},\mbox{ and }
q_k=\frac{p_k+\g_k(p_0+\l)}{1+\l}\mbox{ for }k\geq 1.
\end{equation}
In Table \ref{coup_tab} we give a list of the rates used for the coupling of
$(\hat X,\hat Y)$ to $(X,Y)$. However, before that, we give an intuitive explaination.

We start with equal sizes, $\hat X(0)=X(0)$ and $\hat Y(0)=Y(0)$. 
Each individual in $X(0)$ is paired with
a unique `friend' in $\hat X(0)$, and each individual in $Y(0)$
is paired with
a unique friend in $\hat Y(0)$.  Whenever a cell in $X$ either
multiplies or dies a natural death (transition (i)
in Table \ref{table1}) then its friend in $\hat X$ undergoes
the exact same transition, and the offspring are paired
in the natural way.  Similarly, whenever a cell in $Y$
multiplies (transition (ii)
in Table \ref{table1}) then its friend in $\hat Y$ undergoes
the exact same transition, and again the offspring are paired
in the natural way.  When a cell $Y$ has a lysis
(transition (iii) in Table \ref{table1}), sample a random variable
$\G$ with distribution $(\g_k)_{k\geq0}$.  
Infect $\G \wedge X$ cells from $X$,
but let the friends in $\hat X$ of
the newly infected cells in $X$ remain unchanged (but lose their
friends, existing as singletons).
Proceed by letting the friend in $\hat Y$
of the cell in $Y$ which underwent
lysis be replaced by $\G$ new cells.  Finally, pair the newly
infected cells, now belonging to $Y$, with the new cells of
$\hat Y$.  Note that if $\G>X,$ the this will result in some of the cells
in $\hat Y$ being unpaired.  

Thus every element of $X$ always has a friend in $\hat X$, and
every element of $Y$ always has a friend in $\hat Y$;
but some cells in $\hat X$ and $\hat Y$ might be unpaired.
We let
unpaired cells give rise to independent Markov branching
processes with the correct intensities and offspring distributions.
The rates of the coupled process
$(X,\hat X,Y,\hat Y)$ are summarized in Table~\ref{coup_tab}.
\begin{table}[hbt] 
\centering
\begin{tabular}{l|l|l} 
Transition from $(x,\hat x, y,\hat y)$ to state & Rate & Valid for\\ \hline
 $(x-1+k,\hat x-1+k,y,\hat y)$ & $x p_k$ & $k\geq 0$\\
 $(x,\hat x-1+k,y,\hat y)$ & $(\hat x-x) p_k$ & $k\geq0$\\
 $(x,\hat x,y-1+k,\hat y-1+k)$ & $yp_k$ & $k\geq1$\\
 $(x,\hat x,y,\hat y-1+k)$ & $(\hat y-y) p_k$ & $k\geq1$\\
 $(x-(x\wedge k),\hat x,y-1+(x\wedge k),\hat y-1+k)$ &
$y(p_0+\l)\g_k$ & $k\geq0$\\
 $(x,\hat x,y,\hat y-1+k)$ & $(\hat y-y)(p_0+\l)\g_k$ & $k\geq0$
\end{tabular}
\caption{Transition rates in the coupled chain 
$(X,\hat X,Y,\hat Y)$.
Rates are given for transitions from a state $(x,\hat x,y,\hat y)$ 
and are valid for all $x, \hat{x},y,\hat{y} \geq 0.$  Note that the
ordering $x\leq\hat x$, $y\leq\hat y$ is preserved.}
\label{coup_tab} 
\end{table}
As before, the correct interpretation is to add
the rates of
transitions leading to the same state.  
We note that our coupling satisfies the following:
\begin{enumerate}
\item $X(t)\leq\hat X(t)$ and $Y(t)\leq\hat Y(t)$ for all
$t\geq 0$;
\item if $X(t)\neq0$ then $\hat Y(t)=Y(t)$.
\end{enumerate}

For a probability vector $\pi=(\pi_k:k\geq0)$ we write $\bar\pi$ for
the mean $\sum_{k\geq0}k\pi_k$.
Let $\a$ and $\b$ be the
\emph{Malthusian parameters} of $\hat X$ and
$\hat Y$, respectively, given by
\begin{equation*}
\a=\bar p -1,
\qquad
\b=(\bar q -1)(1+\l).
\end{equation*}
It is well known~\cite{harris63}
that $\hat X(t)/e^{\a t}$ and
$\hat Y(t)/e^{\b t}$
are martingales which converge almost surely to
some nonnegative random variables.  
We have that $P(A\cup B)=1$ where
\begin{equation}\label{eqn1}
A=\big\{\hat X(t)=0\mbox{ for some }t\geq 0\big\}\mbox{ and }
B=\big\{\liminf_{t\rightarrow\infty}\log(\hat X(t))/t>0\big\}.
\end{equation}
Moreover, $P(A)=1$ if and only if $\a\leq 0$.
On the event $B,$ the limit
$\lim_{t\rightarrow\infty}\log(\hat X(t))/t$ 
exists and equals $\a$.  The corresponding statements hold
for $\hat Y(t)$ with $\a$ replaced by $\b$.
Note for future reference that
\begin{equation}\label{beta_alpha}
\begin{split}
\b&=(\bar q -1)(1+\l)
=\Big(\frac{\bar{p}+\bar{\g}(p_0+\l)}{1+\l}-1\Big)(1+\l)\\
&=\bar{p}-1+p_0\bar{\g}+\l(\bar{\g}-1)
=\a+p_0\bar{\g}+\l(\bar{\g}-1).
\end{split}
\end{equation}

Next, let $U$, $V$, $W$ and $\Phi$ denote random variables with the
following distributions.  Firstly, $U$ and $V$ have the distributions
of (the sizes of) $\hat X(1)$ and $\hat Y(1)$, respectively, when
$\hat X(0)=1$ and $\hat Y(0)=1$.  Secondly, $W$ has the distribution of
$\tilde X(1)$, where $\tilde X$ is a branching process, started at 1,
with lifelength intensity 1 and offspring distribution $\pi$ given by
$\pi_0=0$, $\pi_1=p_0+p_1$, and $\pi_k=p_k$ for $k\geq 2$.  Thus
$\tilde X$ is essentially $\hat X$ with deaths suppressed.  Finally,
to define $\Phi$ run a sample of $\hat Y$ for time 1, started with
$\hat Y(0)=1$;  for each branching event that occurs during this time
sample an independent Bernoulli random variable with success
probability $p_0$, and let $L$ denote the total number of successes.
Let $\Phi$ have the distribution of a sum of $L$
independent copies of $\G$.  Thus $\Phi$ is, intuitively,
the number of infection attempts during a unit time interval
starting with one infected cell.

\begin{lemma}\label{expect_lem}
Let $r\geq 1$ and let $D$ denote a random variable with
distribution~$(p_k)_{k \geq 0}$.  Then
\begin{enumerate}
\item $E(U^r)<\infty$ if $E(D^r)<\infty$,
\item $E(V^r)<\infty$ if $E(D^r)<\infty$ and $E(\G^r)<\infty$,
\item $E(W^r)<\infty$ if $E(D^r)<\infty$,
\item $E(\Phi)<\infty$ if $E(\G)<\infty$ and $E(D)<\infty$.
\end{enumerate}
\end{lemma}
\begin{proof}
From 
\cite[Corollary~III.6.1]{athreya_ney}, we know that a branching process
with offspring distribution $\pi$ has finite $r$th moment at time
$t>0$ if $\pi$ has its $r$th moment.  This immediately gives
parts 1 and 3.  Part 2 follows from~\eqref{q_eq}, which
implies that $(q_k)_{k \geq 0}$ has its $r$th moment if $(p_k)_{k \geq 0}$
and $(\g_k)_{k \geq 0}$ do.
For the final part, note that
\begin{equation*}
E(\Phi)=E\Big(\sum_{j\geq 1}\G_j\one\{L\geq j\}\Big)
=E(\G)E(L).
\end{equation*}
An easy (stochastic) upper bound on $L$ is given by $\tilde Y(1)$
where $\tilde Y$ is a branching process with intensity $1+\l$
and offspring distribution $(\tilde q_k)_{k \geq 0}$, where
$\tilde q_0=\tilde q_1=0$, $\tilde q_2=q_2+q_1+q_0$,
and $\tilde q_k=q_k$ for $k\geq 3$.  Thus $E(L)$ is finite
if $E(\G)$ and $E(D)$ are finite, as in part 2.
\end{proof}

We will in what follows always assume that $(p_k)_{k\geq0}$ has finite
second moment, since this is part of the assumptions in
Theorems~\ref{coexist_thm} and~\ref{thmfinexp}.  By
Lemma~\ref{expect_lem} this implies that
$E(U^2)<\infty$,
$E(V^2)<\infty$, $E(W^2)<\infty$ and $E(\Phi)<\infty$.  This
will allow us to apply Chebyshev's bound, which we will use
in the following form.  Let $Z_j$ ($j\geq 1$) be independent,
all with the same nonnegative mean $E(Z)\geq0$
and finite variance $\var(Z)<\infty$ as some random variable $Z$.
Let $N\geq 1$ be any integer and let $\d>0$.
Then
\begin{equation}\label{chebyshev}
\begin{split}
P\Big(\sum_{j=1}^NZ_j>(1+\d)NE(Z)\Big)&\leq
P\Big(\Big[\sum_{j=1}^NZ_j-E(Z_j)\Big]^2>N^2\d^2E(Z)^2\Big)\\
&\leq \frac{N\var(Z)}{N^2\d^2E(Z)^2}
=\frac{1}{N}\cdot\frac{\var(Z)}{\d^2E(Z)^2}.
\end{split}
\end{equation}
Similarly
\begin{equation}\label{chebyshev2}
P\Big(\sum_{j=1}^NZ_j<(1-\d)NE(Z)\Big)\leq
\frac{1}{N}\cdot\frac{\var(Z)}{\d^2E(Z)^2}.
\end{equation}

\subsection{Estimates}

The following lemma says that $Y$ cannot
be much larger than $X$ for very long without
making $X$ extinct.  This lemma will be the main
step in the proof of the case $\b>\a$
in Theorem~\ref{coexist_thm}, which is the case when
the process $\hat Y$ grows much faster than $X$.
In the statement of the  lemma, we let $W$
be as in Lemma~\ref{expect_lem}, and let $\xi$ be a Bernoulli variable
with  success probability $1-e^{-(1-\g_0)(p_0+\lambda)}$
(this being the probability of a lysis leading to at least 
one new infection occuring in a time interval of length 1).  We fix
$c>0$ and let $\d(t)>0$ be any function such that 
\begin{equation}\label{d_eq}
n\d(n)>\frac{1}{2}\log\left(2\frac{E(W)}{E(\xi)}\right)
\end{equation}
for all sufficiently large $n$.
We write
\[
A_n=\{\forall t\in[n,n+1],
0<X(t)\leq e^{(c-\d(t))t}<e^{(c+\d(t))t}\leq Y(t)\}.
\]
\begin{lemma}\label{An_lem}
There is a constant $C>0$ such that for $n$ large enough
that~\eqref{d_eq} holds,
\begin{equation}\label{An_eq}
P(A_n)\leq Ce^{-(c-\d(n))n}.
\end{equation}
In particular, we can take $C=9(\var(W)/E(W)^2+\var(\xi)/E(\xi)^2).$ 
It follows that
$P(A_n\,\mathrm{ i.o.})=0$.
\end{lemma}
Before turning to the proof we remark that we only actually use
this lemma with $\d$ constant.  We prove this slightly more
general result since very little extra work is required, and
we hope that it will be useful for future work.
\begin{proof}
The result is trivial if $\d(n)\geq c$ so we assume that
$\d(n)<c$;  we also assume throughout the proof that
$n$ is large enough that~\eqref{d_eq} holds. Suppose that $A_n$ occurs.
Let $\Phi_n$ denote the number of infection attempts
during the time interval $[n,n+1]$, that is to say the sum of an
independent sample of $\G$ for each lysis of
$(Y(t):t\in[n,n+1])$.
Let $\xi^{(n)}$ be obtained from $(Y(t):t\in[n,n+1])$ as follows.
Start by numbering the elements of $Y(n)$ (arbitrarily);
then observe those elements numbered at most $e^{(c+\d(n))n}$
until they undergo a branching event;  let $\xi_j$ be the indicator
of the event that cell $j$ has a branching event which results
in a lysis for which the associated $\G$-value is at least 1
($\xi_j=0$ if there is no branching event before time $n+1$);
finally let $\xi^{(n)}$ be the sum of the $\xi_j$.
Then $\xi^{(n)}$ has the following properties:
\begin{enumerate}
\item $\xi^{(n)}\leq\Phi_n$,
\item $\xi^{(n)}$ is a sum of $e^{(c+\d(n))n}$ independent
Bernoulli variables, each with success probability
$1-e^{-p_0(1-\g_0)(1+\lambda)}$, and
\item $\xi^{(n)}$ is independent of $(X(t):t\in[n,n+1])$.
\end{enumerate}

Next,
let $W^{(n)}$ denote the total number of healthy cells that ever exist in
the time-interval $[n,n+1]$.
Of course, if $W^{(n)}\leq \Phi_n,$ 
then $A_n$ cannot occur since this would imply that $X(n+1)=0$.
We cannot {\em immediately} conclude from the fact that
$X(t)\leq e^{(c-\d(t))t}$ for every $t\in[n,n+1],$ that $W^{(n)}$
is bounded by $e^{(c-\d(n+1))(n+1)}$.
However
$W^{(n)}$ must be stochastically bounded by
the sum of $e^{(c-\d(n))n}$
independent copies $W_j$
of the random variable $W$ in Lemma~\ref{expect_lem}.
(Recall that $W$ is, intuitively, $\hat X(1)$ when deaths
are suppressed.)
Also, $W^{(n)}$ is independent of $\xi^{(n)}$.
Thus, writing $a_n=e^{(c-\d(n))n}$ and $b_n=e^{(c+\d(n))n}$,
we have that
\begin{equation*}
\begin{split}
P(A_n)&\leq P(W^{(n)}>\xi^{(n)})\leq
P\Big(\sum_{j=1}^{a_n}W_j>\sum_{j=1}^{b_n}\xi_j\Big)\\
&=P\Big(\frac{1}{a_n}\sum_{j=1}^{a_n}W_j>
\frac{b_n}{a_n}\frac{1}{b_n}\sum_{j=1}^{b_n}\xi_j\Big).
\end{split}
\end{equation*}
Note that $b_n/a_n=e^{2n\d(n)}>e^{\log(2E(W)/E(\xi))}=2E(W)/E(\xi)$, 
by~\eqref{d_eq}. 
We get that
\begin{equation*}
\begin{split}
P(A_n) &\leq 
P\Big(\frac{1}{a_n}\sum_{j=1}^{a_n}W_j>
2\frac{E(W)}{E(\xi)}\frac{1}{b_n}\sum_{j=1}^{b_n}\xi_j\Big) \\
 & \leq P\Big(\frac{1}{a_n}\sum_{j=1}^{a_n}W_j>
2\frac{E(W)}{E(\xi)}\frac{2}{3}E(\xi)\Big)
+P\Big(\frac{2}{3}E(\xi)>\frac{1}{b_n}\sum_{j=1}^{b_n}\xi_j\Big) \\
 & \leq \frac{9\var(W)}{a_nE(W)^2}+
\frac{9\var(\xi)}{b_nE(\xi)^2}, 
\end{split}
\end{equation*}
where we use (\ref{chebyshev}) and (\ref{chebyshev2}).
This gives~\eqref{An_eq}.
That $P(A_n\,\mathrm{ i.o.})=0$ follows from the Borel--Cantelli 
lemma.
\end{proof}

Recall that if $U(t)$ is a Markov branching process with 
Malthusian parameter $u$ then $W(t)=U(t)/e^{ut}$ is
a martingale.  
We make no claim as to the originality
of the following lemma, yet have not seen it explicitly
formulated.

\begin{lemma}\label{lemma_expdecay} 
Let $U(t)$ be a branching process whose offspring distribution has
finite second moment and with Malthusian parameter $u>0$.
\begin{enumerate}
\item For any $\Delta>0$ we have that
\begin{equation*} 
P(\exists t\geq0: W(t)\geq \Delta) \leq \Delta^{-1}.
\end{equation*}
\item 
For each $\eps>0$ there is some $\k>0$ such that
\begin{equation*} 
P(\exists t\geq\tau: 0<W(t) <e^{-\eps t}) \leq e^{-\k \tau}.
\end{equation*}
\end{enumerate}
\end{lemma}
\begin{proof}
The first part is simply a consequence of Doob's
submartingale inequality, which gives that for any $T>0$,
\[
P(\exists t\in[0,T]: W(t)\geq\Delta)=
P\Big(\sup_{0\leq t\leq T}W(t)\geq \Delta\Big)
\leq E[W(T)]/\Delta=\Delta^{-1}.
\]
Letting $T \to \infty$ concludes the proof of this case.

For the second part, we proceed by discretizing.  Let $\mu=E[U(1)]=e^u$
and let $W_n=U(n)/\mu^n$ for every $n\in {\mathbb N}$.  
It is no loss of generality to assume that
$\eps<u/2$.  The limit
$W:=\lim_n W_n$ exists a.s. since $(W_n)_{n \geq 1}$ is a
nonnegative  martingale.
A straightforward and standard calculation (see for 
instance~\cite[p.~13]{harris63})  shows that for any $r>n,$
\[
E[\mu^ n(W_r-W_n)^2]=\sigma^ 2(\mu^{-1}+\mu^{-2}+\cdots +\mu^{-r}),
\]
where $\s^2=\var(U(1))$.
Therefore by Fatou's lemma
\begin{equation}
E[(W-W_n)^2]\leq \liminf_{r\rightarrow\oo} E[(W_r-W_n)^2]= 
\frac{\sigma^ 2}{\mu-1}\mu^{-n}
\end{equation}
for all $n$.  Hence by Markov's inequality
\[
P(|W-W_n|>e^{-\eps n}) \leq  \frac{E[(W-W_n)^2]}{e^{-2\eps n}}
\leq \frac{\s^2}{\mu-1}e^{-(u-2\eps) n}.
\]
It is well known (see for instance~\cite[Theorem~8.3]{harris63}) 
that there exists a constant $c_3>0$
such that for any interval $I\subset (0,\infty)$ we have  
$P(W\in I)\leq c_3|I|$. Furthermore, it is also well 
known~\cite[Theorem~8.4]{harris63}
that there exists a constant $c_4>0$ such that 
$P(W=0,W_n\neq0)\leq e^{-c_4 n}$.  Therefore (adjusting $c_3$ as necessary)
\begin{equation} \label{eqn22}
\begin{split}
P(0<W_n<e^{-\eps n})
&\leq P(W=0,W_n>0)+P(0<W<2e^{-\eps n})\\
&\quad+ P(|W-W_n|>e^{-\eps n})\\
&\leq c_3(e^{-c_4 n}+e^{-\eps n}+e^{-(u-2\eps) n}).
\end{split}
\end{equation}
Clearly
\begin{multline}
P(\exists s\geq t: 0<W(s)<e^{-\eps s})\leq
P(\exists n\geq t: 0<W_n<e^{-\eps n/2})\\+
P(\exists s\geq t: 0<W(s)<e^{-\eps s}, \forall n\geq t\;
W_n=0\mbox{ or }W_n\geq e^{-\eps n/2}).
\end{multline}
We have bounded the first probability on the right hand side
in~\eqref{eqn22}.  The second probability is bounded above by 
\begin{equation*}
\begin{split}
P\Big(&\bigcup_{n\geq t}\{\exists s\in[n,n+1]:W(s)<e^{-\eps n},\;
W_n\geq e^{-\eps n/2}\}\Big)\\
&\leq\sum_{n\geq t} P(\exists s\in[n,n+1]:W(s)<e^{-\eps n}\mid
W_n\geq e^{-\eps n/2})P(W_n\geq e^{-\eps n/2})\\
&\leq\sum_{n\geq t} P(\exists s\in[n,n+1]:U(s)<e^{u(n+1)-\eps n}\mid
U(n)\geq e^{un-\eps n/2}).
\end{split}
\end{equation*}
It therefore suffices to show that each of the summands is
exponentially small in $n$ for large enough $n$.  

To establish this we take the following point of view.  Let $M=U(n)$
and label the particles present at time $n$ by $1,2,\dotsc,M$.  If
particle $j$ has a branching event with zero offspring we say that
particle $j$ is \emph{destroyed}.  If it has a branching event with
one or more offspring, we consider particle $j$ to be still present,
essentially identifying it with one of its offspring particles.  With
this convention, we let $A_j$ denote the event that particle $j$ is
ever destroyed during the time interval $[n,n+1]$.  Thus
$P(A_j)<1$ for all $j$, and the events $A_j$ are
independent.  If $U(s)\leq e^{u(n+1)-\eps n}$ for some $s\in[n,n+1]$
then at least $M-e^{u(n+1)-\eps n}$ of the events $A_j$ must occur.
But since $M\geq e^{un-\eps n/2}$
\[
\begin{split}
P\Big(\sum_{j=1}^M\one_{A_j}\geq M-e^{u(n+1)-\eps n} \Big)&\leq
P\Big(\sum_{j=1}^M\one_{A_j}\geq M(1-e^{u}e^{-\eps n/2})\Big)\\
&\leq P\Big(\sum_{j=1}^M\one_{A_j}\geq MP(A_j)(1+\d)\Big)
\end{split}
\]
for large enough $n$ and some $\d>0$.  The latter probability is
by~\eqref{chebyshev} at most
\[
C/M\leq Ce^{-(u-\eps/2)n}.
\]
This gives the result.
\end{proof}

\subsection{A lemma about order statistics}
\label{orderstat_sec}

The following result will be used in the case
$\a>\b$ in Theorem~\ref{coexist_thm}, 
but may also be of independent interest.
The first part essentially goes back
to~\cite{arnold79} (in the case $p=2$),
but we have not found the second part
in the literature.

If $(X_j)_{1 \leq j \leq M}$ is a 
sequence of indentically distributed
random variables, we let
$X_{(1)}\leq X_{(2)}\leq \dotsb \leq X_{(M)}$ denote the 
\emph{order statistics} of $(X_j)_{1 \leq j \leq M}$. 
\begin{lemma}\label{trim_lem}
Let $(X_j)_{1 \leq j \leq M}$
be as above.
\begin{enumerate}
\item If $p>1$ and  $\|X_1\|_p=E[X_1^p]^{1/p}<\oo$ then
for each subset $A\subseteq\{1,\dotsc,M\}$,
\begin{equation}\label{trim_eq}
E\Big[\sum_{j\in A} X_{(j)}\Big]\leq
\|X_1\|_pM^{1/p}m^{1/q},
\end{equation} 
where $m=|A|$ and $1/p+1/q=1$.
\item If the $X_i$ are independent and $E[X_1^2]<\oo$, then
\begin{equation}
\var\Big(\sum_{j=1}^{M-m}X_{(j)}\Big)\leq
\var\Big(\sum_{j=1}^{M}X_{j}\Big)=M\cdot\var(X_1).
\end{equation}
\end{enumerate}
\end{lemma}
\begin{proof}
The first part is a consequence of H\"older's
inequality:
\[
\begin{split}
E\Big[\sum_{j\in A}X_{(j)}\Big]&=
E\Big[\sum_{j=1}^MX_{(j)}\one\{j\in A\}\Big]
\leq E\Big[\sum_{j=1}^M|X_{(j)}|^p\Big]^{1/p}
E\Big[\sum_{j=1}^M\one\{j\in A\}\Big]^{1/q}\\
&=E\Big[\sum_{j=1}^M|X_j|^p\Big]^{1/p} |A|^{1/q}
=\|X_1\|_pm^{1/q}M^{1/p}.
\end{split}
\]
For the second part,
let $X$ denote the sequence $(X_1,\dotsc,X_M)$ and let
\[
f(X)=\sum_{j=1}^{M-m} X_{(j)}\quad\mbox{and}\quad
g(X)=\sum_{j=M-m+1}^{M} X_{(j)}.
\]
Note that both $f$ and $g$ are increasing functions in the sense that if 
$x=(x_1,\ldots,x_n)$ and $y=(y_1,\ldots,y_n)$ satisfy $x_i \leq y_i$
for every $i=1,\ldots,n,$ then $f(x)\leq f(y)$ and $g(x)\leq g(y)$.
Thus also $f(X)-E[f(X)]$ and $g(X)-E[g(X)]$ are increasing
functions.  It follows from Harris' inequality
that 
\begin{multline*}
E\big[(f(X)-E[f(X)]) (g(X)-E[g(X)])\big]\\\geq
E\big[f(X)-E[f(X)]\big]E\big[g(X)-E[g(X)]\big]=0,
\end{multline*}
that is to say
\[
\mathrm{Cov}\Big(\sum_{j=1}^{M-m} X_{(j)},\sum_{j=M-m+1}^{M} X_{(j)}\Big)\geq0.
\]
It follows that 
\[
\begin{split}
\var\Big(\sum_{j=1}^MX_j\Big)&=
\var\Big(\sum_{j=1}^{M-m}X_{(j)}+\sum_{j=M-m+1}^{M}X_{(j)}\Big)\\
&=\var\Big(\sum_{j=1}^{M-m}X_{(j)}\Big)+
\var\Big(\sum_{j=M-m+1}^{M}X_{(j)}\Big) \\
& \ \ + 2 \mathrm{Cov}\Big(\sum_{j=1}^{M-m} X_{(j)},\sum_{j=M-m+1}^{M}
X_{(j)}\Big)\\
&\geq \var\Big(\sum_{j=1}^{M-m}X_{(j)}\Big).
\end{split}
\]
\end{proof}

Setting $m=1$ in~\eqref{trim_eq} we deduce 
that $E[X_{(M)}]$ is of order at most $M^{1/p}$
when the $X_i$ have finite $p$:th moment.  
Results of this type, usually formulated 
for $p=2$, go back 
to~\cite{arnold79,gumbel54,hartle54}

Note that~\eqref{trim_eq}
is in some sense sharpest when
$A=\{M-m+1,\dotsc,M\}$ because then the sum consists of the $m$
largest terms;  this is the case we will be using.

\section{Proof of Theorem~\ref{coexist_thm}} \label{sec_thm1}

Clearly (by~\eqref{eqn1}) 
$\zeta=0$ if either $\a\leq0$ or $\b\leq0$, so we assume
henceforth that $\a,\b>0$.
The proof of Theorem~\ref{coexist_thm} will be divided into the three 
cases (i) $\a<\b$, (ii) $\a=\b$ and (iii) $\a>\b$.  

\begin{proof}[The case $\a<\b$]
The intuition is that if coexistence
were to take place, then $Y(t)$ would eventually be \emph{much}
larger than $X(t)$;  but then there is a good
chance that all healthy cells are infected in, say, time 1, which
would contradict coexistence.
To make this intuition exact, let $c=(\a+\b)/2$, $\d=(\b-\a)/4>0$, 
and use~(\ref{eqn1}) to see that
\begin{equation}\label{tu_eq}
\begin{split}
P(T_\ru=\infty)&=
P(T_\ru=\infty, \exists t_0: \hat X(t)\leq e^{(c-\d)t}
<e^{(c+\d)t}\leq \hat Y(t), \  \forall t\geq t_0)\\
&\leq P(\exists t_0: 0<X(t)\leq e^{(c-\d)t}
<e^{(c+\d)t}\leq Y(t), \  \forall t\geq t_0),
\end{split}
\end{equation}
since on $\{T_\ru=\infty\}$
we have that $0<X(t)\leq \hat X (t)$
and $Y(t)=\hat Y(t)$ for every $t\geq 0$.
Trivially, the right hand side is bounded above by
$P( A_n \textrm{ i.o.})$ where
\[
A_n:=\{\forall t\in[n,n+1], 0<X(t) \leq e^{(c-\d)t}
<e^{(c+\d)t}\leq Y(t)\}.
\]
But $P( A_n \textrm{ i.o.})=0$ by Lemma~\ref{An_lem}.
\end{proof}

\begin{proof}[The case $\a=\b$]
For the case $\a=\b>0$ the intution is that there will typically be so
many infection events that $X$ {\em effectively} 
(i.e. counting losses due to infections) 
has a strictly larger rate of deaths
than $\hat X$, allowing us to essentially reduce this case to the
case $\a<\b$.  
Note that the process 
\[
R(t)=\frac{\hat Y(t)}{\hat X(t)}
=\frac{\hat Y(t)}{e^{\alpha t}}\frac{e^{\alpha t}}{\hat X(t)}
\]
converges almost surely to some random variable $R$,
since $\hat Y(t)/e^{\alpha t}$
and $\hat X(t)/e^{\alpha t}$ are nonnegative martingales.
The limit $R$ may be infinite, but on the event $\{T_\ru=\oo\}$
we have that $0<R<\oo$.  Furthermore, since $0$ is an
absorbing state for the process $R(t)$ we have (up to a null event) that
$\{T_\ru=\oo\}\seq\{\inf_{t\geq 0}R(t)>0\}$.  It follows that for 
each $r>0$ we have 
\[
\{T_\ru=\oo\}\seq \{0<\inf_{t\geq 0}R(t)<r\}\cup G_r,
\]
where 
\[
G_r=\Big\{\frac{\hat Y(t)}{\hat X(t)}\geq r \ \forall t\geq0\Big\} 
\cap \{X(t)Y(t)>0 \ \forall t\geq0\}.
\]
For each $\d>0$ we may choose $r>0$ sufficiently small so that
$P(0<\inf_{t\geq 0}R(t)<r)\leq\d$ and thus
$\zeta\leq \d+P(G_r)$.  We aim to show that $P(G_r)=0$
for each $r>0$;  since $\d>0$ was arbitrary this will complete
the proof.

Fix $\d,r>0$ as above.
We will couple $X$, $\hat X$ and $Y$ to a new
process $X'$ which is obtained by taking into account
some of the effect of $Y$ on $X$.  The process $X'(t)$
will be a Markov branching process and will satisfy 
$X'(t)\leq \hat X(t)$ for all $t\geq 0$.  Moreover,
on the event $G_r$
we will have that $X(t)\leq X'(t)$ for all $t\geq 0$.
We let $X'(0)=X(0)$.
The rates governing the quadruple $(X,X',\hat X,Y)$ are given
in Table~\ref{coex_tab}, where we have written 
\[
\kappa=\kappa(x',y)=\Big(r\frac{x'}{y}\Big)\wedge 1.
\]

\begin{table}[hbt]
\centering
\begin{tabular}{l|l|l} 
Transition to state & Rate & For\\ \hline
 $(x+k-1,x'+k-1,\hat x+k-1,y)$ & $(x\wedge x') p_k$ & $k\geq 0$\\
 $(x+k-1,x',\hat x+k-1,y)$ & $(x-x\wedge x') p_k$ & $k\geq0$\\
 $(x,x'+k-1,\hat x+k-1,y)$ & $(x'-x\wedge x') p_k$ & $k\geq0$\\
 $(x,x',\hat x+k-1,y)$ & $(\hat x-x\vee x') p_k$ & $k\geq0$\\
 $(x,x',\hat x,y+k-1)$ & $y p_k$ & $k\geq1$\\
$(x,x',\hat x,y-1)$ & $y (p_0+\l)\g_0$ & \\
 $(x-(x\wedge k),x',\hat x,y-1+(x\wedge k))$ &
$y(p_0+\l)\g_k(1-\kappa)$ & $k\geq1$\\
 $(x-(x\wedge k),x'-1,\hat x,y-1+(x\wedge k))$ &
$y(p_0+\l)\g_k \kappa$ & $k\geq1$\\
 $(x,x'-1,\hat x,y)$ & $(rx'-\kappa y)(p_0+\l)(1-\g_0)$ & 
\end{tabular}
\caption{Transition rates in the coupled chain 
$(X,X',\hat X,Y)$.
Rates are given for transitions from a state $(x,x',\hat x,y)$ and 
are valid for all $x,x',\hat x,y \geq 0.$}
\label{coex_tab} 
\end{table}

We note from Table~\ref{coex_tab} that the triple
$(X,\hat X,Y)$ has the correct marginal distribution, i.e.\! as 
described in Section~\ref{auc_sec}.
For example, summing the first
two lines the of the table gives the rate 
$(x\wedge x'+x-x\wedge x')p_k=xp_k$
for the transition
 $x\rightarrow x+k-1$.  
Similarly, $\hat x\rightarrow\hat x+k-1$ at rate given by the sum of
the first four lines, and using that 
$x+x'-x\wedge x'=x\vee x'$ we get the correct rate $\hat xp_k$.

Consider now the marginal distribution for $X'$.  First note
that, since $\kappa\leq rx'/y$, the final rate in Table~\ref{coex_tab}
is nonnegative.  Adding the rates for the transitions
$x'\rightarrow x'-1$, we find that this transition occurs
at rate
\[
x'(p_0+r(p_0+\l)(1-\g_0)).
\]
Together with the rates for $x'\to x'+k-1$ for $k\geq 1$, this 
means that $X'(t)$ is a Markov branching process with
lifelength intensity
\[
1+r(p_0+\l)(1-\g_0)
\]
and offspring distribution $p'$ given by
\begin{equation*}
\begin{split}
p'_0&=\frac{p_0+r(p_0+\l)(1-\g_0)}{1+r(p_0+\l)(1-\g_0)},\\
p'_k&=\frac{p_k}{1+r(p_0+\l)(1-\g_0)},
\quad k\geq 1.
\end{split}
\end{equation*}
In particular, the Malthusian parameter of $X'$ is
\begin{equation*}
\begin{split}
\a'&=(1+r(p_0+\l)(1-\g_0))(\bar{p'}-1)\\
&=\a-r(p_0+\l)(1-\g_0)\\
&<\a\quad\mbox{for }r>0.
\end{split}
\end{equation*}
Clearly $X'(t)\leq\hat X(t)$ for all $t\geq0$.
On the event $G_r$ we also have that $\hat Y(t)= Y(t)$
for all $t\geq0$ and that $\hat Y(t)/\hat X(t)\geq r$
for all $t\geq0$.  It follows that, on $G_r$, we have that
\[
r\leq \frac{\hat Y(t)}{\hat X(t)}=
\frac{Y(t)}{\hat X(t)}\leq \frac{Y(t)}{X'(t)}\quad
\mbox{for all } t\geq 0,
\]
so that $rX'(t)/Y(t)\leq1$ and hence 
$\kappa(X'(t),Y(t))=rX'(t)/Y(t)$.  Thus the final rate
in Table~\ref{coex_tab} is always 0 on the event
$G_r$, and hence so is the second rate. Therefore, we get that $G_r\seq\{X(t)\leq X'(t)\,\forall t\geq 0\}$.

Let $c=(\a+\a')/2$ and $\d=(\a-\a')/4>0$.   Using~\eqref{eqn1} we therefore
deduce that
\begin{equation}\label{r3_eq}
P(G_r)\leq P(\exists t_0: 0<X(t)\leq X'(t)\leq e^{(c-\d)t}
<e^{(c+\d)t}\leq Y(t), \  \forall t\geq t_0).
\end{equation}
By Lemma~\ref{An_lem}, the probability on the right equals zero.
Since $\d>0$ was arbitrary it follows that $\zeta=0$.
\end{proof}

We are now ready to prove the final case of Theorem \ref{coexist_thm}.

\begin{proof}[The case $\a>\b$]
The intuition here is that $X(t)$
`wants' to be of the order $e^{\a t}$ and $Y(t)$ `wants' to be of the,
much smaller, order $e^{\b t}$.  Typically, therefore, the infection
will have very little impact on the healthy population.

To make this intuition rigorous, let
\begin{equation*}
a_n=\prod_{k=2}^n\Big(1-\frac{2}{k^2}\Big),\quad
b_n=\prod_{k=2}^n\Big(1-\frac{1}{k^2}\Big),\quad
c_n=\prod_{k=2}^n\Big(1+\frac{1}{k^2}\Big).
\end{equation*}
Note that $a_n$ and $b_n$ form decreasing sequences with limits in
$(0,1)$ and that $c_n$ is an increasing sequence with limit in
$(1,\oo)$. 
Write $B_n$ for the event that
\begin{equation*}
X(n)\geq a_n e^{\a n}\mbox{ and }
b_n e^{\b n}\leq Y(n)\leq c_n e^{\b n}.
\end{equation*}
We will prove that there is some $N$ such that
\begin{equation} \label{eqn_Bn}
P(B_{n+1}\mid B_n)\geq 1-\frac{3}{n^2}
\end{equation}
for all $n\geq N$.  This  will, using the Markov property, establish
the result, since $P(B_N)>0$ and 
\[
\zeta\geq P(\cap_{n\geq N} B_n)=P(B_N)\prod_{n\geq N}P(B_{n+1}\mid B_{n})>0.
\]

We start by observing that (again using that $Y(t)=\hat Y(t)$ whenever $X(t)>0$)
\begin{eqnarray} \label{eqn_3parts}
\lefteqn{P(B_{n+1} \mid B_n)}\\
& & =P(X(n+1)\geq a_{n+1} e^{\a (n+1)}, b_{n+1}e^{\b(n+1)}\leq \hat Y(n+1)\leq c_{n+1}e^{\b(n+1)} \mid B_n)\nonumber  \\
& & \geq 1-P(X(n+1)<a_{n+1} e^{\a (n+1)}\mid B_n) 
-P\big(\hat Y(n+1) < b_{n+1}e^{\b(n+1)} \mid B_n \big) \nonumber \\
& & \ \ -P\big(\hat Y(n+1)> c_{n+1}e^{\b(n+1)}\mid B_n \big). \nonumber
\end{eqnarray}

We will proceed to show that all three probabilities on the right hand side are small.
To prove that $P(X(n+1)<a_{n+1} e^{\a (n+1)}\mid B_n)$ is small, 
let $\Phi_n$ denote the number of infection attempts during the time
interval $[n,n+1]$, as in the proof of 
Lemma~\ref{An_lem}.  We will first show that $\Phi_n$
will typically be much smaller than $X(n)$, and will deduce from this
the required lower bound on $X(n+1)$.  For the bound on $\Phi_n$, we
use Markov's inequality to see that
\begin{equation}\label{m2}
P(\Phi_n\geq c_ne^{\b n}\cdot (n+1)^2E(\Phi)\mid B_n)
\leq \frac{E(\Phi_n\mid B_n)}{c_ne^{\b n}\cdot (n+1)^2E(\Phi)}
\leq \frac{1}{(n+1)^2},
\end{equation}
where $\Phi$ is the random variable of Lemma~\ref{expect_lem} and we
used the fact that, given $B_n$, the number $\Phi_n$ of infection attempts is
dominated by the sum of $c_n e^{\b n}$ independent copies of $\Phi$.

Let $M=M(n)=a_ne^{\a n}$ and $m=m(n)=c_ne^{\b n}(n+1)^2E(\Phi)$
(so $X(n)\geq M$ on $B_n$, and $m$ is the quantity
in~\eqref{m2}).
Let $(U_j)_{1\leq j \leq M}$ denote independent copies of the random variable $U$ of
Lemma~\ref{expect_lem}. The lower bound on $X(n+1)$ will be
obtained by noting that the impact of infection during the
time interval $[n,n+1]$ can be no larger than the effect of removing,
at time $n$, those $\Phi_n$ healthy cells that would otherwise give
rise to the largest ancestry at time $n+1$.  In particular,
\begin{equation}
X(n+1)\geq \sum_{j=1}^{X(n)-\Phi_n} U_{(j)},
\end{equation}
where $U_{(1)}\leq U_{(2)}\leq\cdots\leq U_{(M)}$
denote the order statistics of $U_1,\dotsc,U_{M}$ as in
Section~\ref{orderstat_sec}.
  For $n$ large
enough we have $M\geq m$, and on the event
$B_n\cap\{\Phi_n\leq m\}$ we have
\begin{equation}
X(n+1)\geq \sum_{j=1}^{M-m} U_{(j)}.
\end{equation}
Recall that $E(U_j)=e^\a$.
From the first part of Lemma~\ref{trim_lem}, we have that 
\[
E\Big[\sum_{j=M-m+1}^{M}U_{(j)}\Big]=
O\Big(\sqrt{mM}\Big)
=O\left(M\frac{n}{e^{(\a-\b)n/2}}.
\right)
\]
Observe that $ a_{n+1}e^{\a (n+1)}=\big(1-\tfrac{2}{(n+1)^2}\big)M\cdot e^\a$
and that for large enough $n$ we have that
\begin{equation}\label{l1}
\begin{split}
P\Big(\sum_{j=1}^{M-m}U_{(j)}&<
\big(1-\tfrac{2}{(n+1)^2}\big)M\cdot e^\a\Big)=
P\Big(\sum_{j=1}^{M-m}U_{(j)}-E\Big[\sum_{j=1}^{M}U_{j}\Big]<
-\tfrac{2Me^\a}{(n+1)^2}\big)\Big)\\
&=P\Big(\sum_{j=1}^{M-m}U_{(j)}-E\Big[\sum_{j=1}^{M-m}U_{(j)}\Big]<
-\tfrac{2Me^\a}{(n+1)^2}+E\Big[\sum_{j=M-m+1}^{M}U_{(j)}\Big]\Big)\\
&\leq
P\Big(\sum_{j=1}^{M-m}U_{(j)}-E\Big[\sum_{j=1}^{M-m}U_{(j)}\Big]<
-\tfrac{Me^\a}{(n+1)^2}\Big).
\end{split}
\end{equation}
By Chebyshev's bound~\eqref{chebyshev} and the first part of 
Lemma~\ref{trim_lem},
\[
\begin{split}
P\Big(\sum_{j=1}^{M-m}U_{(j)}-E\Big[\sum_{j=1}^{M-m}U_{(j)}\Big]<
-\tfrac{Me^\a}{(n+1)^2}\Big)&\leq
\frac{(n+1)^4\var\Big(\sum_{j=1}^{M-m}U_{(j)}\Big)}{e^{2\a}M^2}\\
&\leq \frac{(n+1)^4\var(U_1)}{e^{2\a}M}=O(e^{-\a n}).
\end{split}
\]
Taking into account also~\eqref{m2}
it follows that
\begin{eqnarray*}
\lefteqn{P\big(X(n+1)\geq a_{n+1}e^{\a(n+1)}\mid B_n\big)}\\
& & \geq P\big(X(n+1)\geq a_{n+1}e^{\a(n+1)} \mid \Phi_n \leq m, B_n\big)
P\big(\Phi_n \leq m \mid B_n\big) \\
& & \geq P\Big(\sum_{j=1}^{M-m} U_{(j)}  \geq a_{n+1}e^{\a(n+1)} \Big)
\Big(1-\frac{1}{(n+1)^2}\Big) \\
& & \geq \big(1-O(e^{-\alpha n})\big)\Big(1-\frac{1}{(n+1)^2}\Big)
\geq 1-\frac{2}{(n+1)^2},
\end{eqnarray*}
for $n$ large enough.

We proceed with the second and third probabilities on the right hand side of (\ref{eqn_3parts}).
We have, with $V_j$ independent and having the distribution of $V$ 
in Lemma \ref{expect_lem}, using that $E(V)=e^{\b}$, (\ref{chebyshev2}) and that $Y(t)=\hat Y(t)$
whenever $X(t)>0,$
\begin{eqnarray*}
\lefteqn{P\big(\hat Y(n+1) < b_{n+1}e^{\b(n+1)} \mid B_n \big)}\\
& & =P\big(\hat Y(n+1) < b_{n+1}e^{\b(n+1)} \mid X(n)\geq a_n e^{\a n},
b_{n}e^{\b n}\leq\hat Y(n)\leq c_{n}e^{\b n}\big) \\
& & =P\big(\hat Y(n+1) < b_{n+1}e^{\b(n+1)} \mid 
b_{n}e^{\b n}\leq\hat Y(n)\leq c_{n}e^{\b n}\big) \\
& & \leq P\big(\hat Y(n+1) < b_{n+1}e^{\b(n+1)} \mid 
\hat Y(n)=b_{n}e^{\b n}\big) \\
& & =P\Big(\sum_{j=1}^{b_ne^{\b n}}V_j < b_{n+1}e^{\b(n+1)} \Big) \\
& & =P\Big(\sum_{j=1}^{b_ne^{\b n}}V_j< \Big(1-\frac{1}{(1+n)^2}\Big)b_ne^{\b n} E(V) \Big) \\
& & \leq \frac{1}{b_ne^{\b n}}\frac{(1+n)^4\var(V)}{E(V)^2}=O(e^{-\b n}).
\end{eqnarray*}
Similarly, but using (\ref{chebyshev}) in place of (\ref{chebyshev2}) ,
\begin{eqnarray*}
\lefteqn{P\big(\hat Y(n+1)> c_{n+1}e^{\b(n+1)}\mid B_n \big)}\\
& &  \leq P\big(\hat Y(n+1)> c_{n+1}e^{\b(n+1)}\mid \hat Y(n)=c_{n}e^{\b n}\big) \\
& & =P\Big(\sum_{j=1}^{c_ne^{\b n}}V_j > c_{n+1}e^{\b(n+1)} \Big)\\
& & =P\Big(\sum_{j=1}^{c_ne^{\b n}}V_j >\Big(1+\frac{1}{(1+n)^2}\Big)c_{n}e^{\b n}E(V) \Big) \\
& & \leq \frac{1}{c_ne^{\b n}}\frac{(1+n)^4\var(V)}{E(V)^2}=O(e^{-\b n}).
\end{eqnarray*}
We conclude that (\ref{eqn_Bn}) holds for $n$ large enough.
\end{proof}

\section{Proof of Theorem~\ref{thmfinexp}} \label{sec_thm2}

The proof of Theorem~\ref{thmfinexp}
will be in two parts.
\begin{proof}[The case $\alpha<0$]
It is well known (see~~\cite[Theorem~11.1]{harris63}) 
that the probability that
a subcritical branching process survives until time $t>0$ 
decays exponentially fast in $t$. That is, there exists
$c>0$ such that for every $t>0,$
\[
P(X(t)>0)\leq e^{-ct}.
\]
Letting $T_X=\inf\{t:X(t)=0\}$ we get that
$E[T_\ru]\leq E[T_X]<\oo$.
\end{proof}

\begin{proof}[The case $0<\alpha<\beta$]
Similarly to (\ref{tu_eq}) let $c=(\a+\b)/2$ and
$\d=(\b-\a)/4$, and note that $c-\d=\a+\d$ and 
$c+\d=\b-\d$.  We have that 
\begin{eqnarray} \label{eqn3}
\lefteqn{P(T_\ru \geq \tau)}\\
& & =P(T_\ru \geq \tau, \hat X(t)\leq e^{(c-\d)t}
<e^{(c+\d)t}\leq \hat Y(t), \  \forall t\geq \tau/2)\nonumber\\
& & \ \ +P(T_\ru \geq \tau, \{ \hat X(t)\leq e^{(c-\d)t}
<e^{(c+\d)t}\leq \hat Y(t), \  \forall t\geq \tau/2\}^c)\nonumber\\
& & \leq P(T_\ru \geq \tau, \hat X(t)\leq e^{(c-\d)t}
<e^{(c+\d)t}\leq \hat Y(t), \  \forall t\geq \tau/2)\nonumber\\
& & \ \ +P(T_\ru \geq \tau, \exists t\geq \tau/2: \hat X(t) \geq e^{(\a+\d)t}
\mbox{ or } \exists t\geq \tau/2: 0<\hat Y(t)<e^{(\b-\d)t})\nonumber\\
& & \leq P(T_\ru \geq \tau, \hat X(t)\leq e^{(c-\d)t}
<e^{(c+\d)t}\leq \hat Y(t), \  \forall t\geq \tau/2)\nonumber\\
& & \ \ +P(\exists t\geq \tau/2: \hat X(t) \geq e^{(\a+\d)t}
\mbox{ or } \exists t\geq \tau/2:0<\hat Y(t)<e^{(\b-\d)t}).\nonumber
\end{eqnarray}
For the first part of the right hand side of (\ref{eqn3}),
we consider (for simplicity) first the case $\tau=2n,$ where we get
\begin{eqnarray*}
\lefteqn{P(T_\ru \geq \tau, \hat X(t)\leq e^{(c-\d)t}
<e^{(c+\d)t}\leq \hat Y(t), \  \forall t\geq \tau/2)}\\
& & =P(T_\ru \geq 2n, \hat X(t)\leq e^{(c-\d)t}
<e^{(c+\d)t}\leq \hat Y(t), \  \forall t\geq n) \\
& & \leq P( 0<X(t)\leq e^{(c-\d)t}
<e^{(c+\d)t}\leq Y(t), \  \forall t\in [n,n+1])=P(A_n),
\end{eqnarray*}
where $A_n$ is as in Lemma~\ref{An_lem}. 
According to that lemma, there exists a $c_2>0$ such that for any $n,$
we have that $P(A_n) \leq e^{-2 c_2 n}=e^{-c_2 \tau}.$
It is easy to see that the same holds for all 
$\tau$ (adjusting $c_2$ if necessary).

For the second part of the right hand side of (\ref{eqn3}), we use 
Lemma \ref{lemma_expdecay}, to conclude that there exists a 
$c_1=c_1(\d)>0$ such that for any $\tau,$
\begin{eqnarray*}
\lefteqn{P(\exists t\geq \tau/2: \hat X(t) \geq e^{(\a+\d)t} \cup
\exists t\geq \tau/2:0<\hat Y(t)<e^{(\b-\d)t})}\\
& & \leq P(\exists t\geq \tau/2: \hat X(t) \geq e^{\a t+\d \tau/2})
+P(\exists t\geq \tau/2:0<\hat Y(t)<e^{(\b-\d)t})\\
& & \leq e^{-\d \tau/2}+e^{-c_1\tau }.
\end{eqnarray*}

We conclude that there exists $c_3>0$ such that 
$P(T_\ru \geq t)\leq e^{-c_3t}$ for any $t>0,$
and so $E[T_\ru]<\oo$.
\end{proof}

\begin{remark}\label{inf_rk}
Clearly $E[T_\ru]=\oo$ when $\a>\b>0$, since then $T_\ru$
takes value $\oo$ with positive probability.
We have not been able to determine in general
whether or not $E[T_\ru]$
is finite in the remaining case $\a=\b$, but in the following
special case it is easily seen to be finite.
Suppose $\a=\b=0$, $\g_0=1$ and $\l=0$. Then $X$ and $Y$ form
independent critical branching processes.  The extinction
times $T_X$ and $T_Y$ for these respective processes
satisfy
\[
P(T_X>t)\sim\frac{1}{t},\quad
P(T_Y>t)\sim\frac{1}{t};
\]
see~\cite[p.~159]{athreya_ney}.  Thus $T_\ru=\min\{T_X,T_Y\}$
satisfies
\[
P(T_\ru>t)=P(T_X>t)P(T_Y>t)\sim\frac{1}{t^2}
\]
so
\[
E[T_\ru]=\int_0^\oo P(T_\ru>t)dt\leq 1+\int_1^\oo P(T_\ru>t)dt
\sim1+\int_1^\oo\frac{dt}{t^2}<\oo.
\]
\end{remark}

\section{Applications of the main results} \label{sec_app}

In this section we will briefly discuss some applications of our main theorems.
Using our results on coexistence 
we are able to comment more on the issue of extinction of $Y$,
which was the main focus of~\cite{BBBN}.

Central to the analysis in the present article were the auxiliary processes
$\hat X$ and $\hat Y.$ Recall that $\hat Y$ was in essence 
the process $Y$ in an `infinite
sea of food', i.e. $X(0)=\infty.$ However, if instead $X(t)=0,$ 
then $(Y(t+s))_{s\geq 0}$ has no healthy cells
to feed on, and therefore $(Y(t+s))_{s\geq 0}$ grows at the exponential rate 
(see also (\ref{beta_alpha}))
\[
\beta'=\bar p-1-\lambda=\alpha-\lambda.
\]
The qualitative behavior of $(X(t),Y(t))_{t\geq 0}$ depends on the values of
$\alpha,\beta$ and $\beta'.$ We discuss the possible different regimes. \\

\noindent{\bf Regime 1.} If $\alpha\leq 0$ then $(X(t))_{t\geq 0}$ eventually dies out,
and since $\beta' \leq \alpha,$ so does $(Y(t))_{t\geq 0}.$ Hence
$\eta=1$. \\

\noindent{\bf Regime 2.} 
If $0<\alpha \leq \beta$ then if $\gamma_0>0$ it might be the case
that $(Y(t))_{t\geq 0}$ dies out spontaneoulsy. However, if it does not, then 
according to Theorem~\ref{coexist_thm}, instead $(X(t))_{t\geq 0}$ will go extinct. 
If $\beta'\leq 0,$ we then conclude that also
$(Y(t))_{t\geq 0}$ dies out, that is $\eta=1$. 
However, if $\beta'> 0$ then $(Y(t))_{t\geq 0}$ can survive
on its own, that is $\eta<1$.\\

\noindent{\bf Regime 3.} 
If $0<\beta<\alpha$ we are in the coexistence regime, in particular
$\eta<1$.   As 
stated in Theorem~\ref{coexist_thm}, it might be the case that $X(t)Y(t)>0$ 
for all $t>0.$ However, as in Regime 2, if $\gamma_0>0,$ it is possible that 
$(Y(t))_{t\geq 0}$ dies out. Furthermore, if $(X(t))_{t\geq 0}$ dies out, 
then the behavior of $(Y(t))_{t\geq 0}$ would again be governed by the sign 
of $\beta'.$\\

\noindent {\bf Regime 4.} If $\beta<0$ then $(Y(t))_{t\geq 0}$
eventually dies out, that is $\eta=1$.\\

We can draw qualitative conclusions from the above description,
using also~\eqref{beta_alpha}.
For instance, if we fix $\alpha>0$
and $E(\G)\geq 1$ it follows that $\alpha \leq \beta$ for 
every $\lambda\geq 0,$ and so we are always
in Regime 2. As long as
$\lambda$ is small enough, so that $\beta'>0,$ the process $(Y(t))_{t\geq 0}$ can survive.
This supports the intuition that small $\lambda$ is good for the long term survival
of $(Y(t))_{t\geq 0},$ see \cite{BBBN}.

If instead $\alpha>0$ while $E(\G)< 1$ we see that we are in Regime 2
for  small values of $\lambda$
and in Regime 3 for large values of $\lambda.$ 
Depending on the exact
values of $\alpha$ and $E(\G)$ we have the following possibilities:
\begin{itemize}
\item for small $\lambda$ we have $0<\alpha < \beta,$ and $\beta'>0$
so that $(Y(t))_{t\geq 0}$ might survive, that is $\eta<1$;
\item for slightly larger $\lambda$ we can have $0<\alpha < \beta,$
and $\beta'\leq0$ so that $(Y(t))_{t\geq 0}$ dies out, that is $\eta=1$;
\item for larger $\lambda$ we have $0<\beta<\alpha,$ so that
  $(Y(t))_{t\geq 0}$ might again survive, that is $\eta<1$;
\item for even larger $\lambda$ we have $\beta\leq 0$ so that
  $(Y(t))_{t\geq 0},$ again dies out, that is $\eta=1$.
\end{itemize}
In \cite{BBBN}, monotonicity of $\eta$ as a function 
of $\lambda$ was established when $\g_0=0.$
In contrast, we see here that monotonicity of $\eta$ in $\l$ may fail if
$E(\G)<1$ (and it is easy to find specific parameters 
for this to be the case). 
Note also the difference between the first case, in which $(Y(t))_{t\geq 0}$
is strong enough to survive on its own, and case three where $(Y(t))_{t\geq 0}$
{\em needs} the process $(X(t))_{t\geq 0}$ to feed on.

\bibliography{vip}
\bibliographystyle{plain}

\end{document}